\documentclass{amsart}

\usepackage[left=3cm,right=3cm,top=2cm,bottom=2cm]{geometry} % page settings
\usepackage{amsmath} 
\usepackage{mathtools}
\usepackage[utf8]{inputenc}
\usepackage{amssymb,amsthm, amsmath, amsfonts, amsbsy}
\usepackage{hyperref}
\usepackage[all]{xy}
\usepackage{enumerate}
\usepackage[mathscr]{eucal}
\usepackage{bbm}
\theoremstyle{plain}
\newtheorem{theorem}{Theorem}[section]
\newtheorem{lemma}[theorem]{Lemma}
\newtheorem{proposition}[theorem]{Proposition}

\newtheorem{corollary}[theorem]{Corollary}
\numberwithin{equation}{section}

\theoremstyle{definition}

\newtheorem{definition}[theorem]{Definition}

\newtheorem{remark}[theorem]{Remark}
\newtheorem{notation}[theorem]{Notation}

\DeclareMathOperator{\Mod}{-Mod}
\DeclareMathOperator{\module}{-mod}
\DeclareMathOperator{\fdmod}{-fdmod}
\DeclareMathOperator{\Hom}{Hom}

\def\mc{\mathcal{C}} % the symbol for category FI should be a mathcal C.
\def\mcc{\mc^m} % the symbol for category FI^m should be C^m.
\def\mcct{\mcc_{\leq \mathbf{t}}}

\def\kmcc{k\mcc}

\def\obj{\operatorname{obj}}
\def\ce{:=}

\def\im{\operatorname{im}}
\def\to{\longrightarrow}

\def\Ext{\operatorname{Ext}}

\def\coind{Q}
\def\kMod{k\text{-Mod}}

\newcommand{\bff}{{\mathbf{f}}}

\newcommand{\FI}{{\mathrm{FI}}}

\newcommand{\FIM}{{\mathrm{FI}^m}}

\newcommand*{\Mor}[1]{\operatorname{Mor}(#1)}

\newcommand*{\cce}[1]{M(#1)}

\newcommand*{\arr}[1]{\mathbf{#1}} % \arr{a} means object $a$ in FI^m
\newcommand*{\numset}[1]{[#1]}
\newcommand*{\set}[1]{\{#1\}}
\newcommand*{\homo}[3]{\operatorname{Hom}_{#1}(#2,#3)}

\title{Locally self-injective property of $\FIM$}
\author{Duo Zeng}
\address{LCSM (Ministry of Education), School of Mathematics and Statistics, Hunan Normal University, Changsha, Hunan 410081, China.}
\email{zengduo@hunnu.edu.cn}

\thanks{The author is partially supported by the National Natural Science Foundation of China (Grant No. 11771135) of his advisor Liping Li. He would like to thank Prof. Li who led the author into this area. Without his invaluable suggestions and help this paper would not have appeared.}

\begin{document}

\begin{abstract}
In this paper we consider the locally self-injective property of the product $\FIM$ of the category $\FI$ of finite sets and injections. Explicitly, we prove that the external tensor product commutes with the coinduction functor, and hence preserves injective modules. As corollaries, every projective $\FIM$-modules over a field of characteristic 0 is injective, and the Serre quotient of the category of finitely generated $\FIM$-modules by the category of finitely generated torsion $\FIM$-modules is equivalent to the category of finite dimensional $\FIM$-modules.
\end{abstract}

\maketitle



\section{Introduction}

\subsection{Motivation}
Representation theory of infinite combinatorial categories has attracted much attention with its rich applications in studying (co)homological groups of topological spaces, geometric groups, and algebraic varieties. One of the most important infinite combinatorial categories is the category $\FI$ of finite sets and injections, whose representation theoretic and homological properties are extensively studied as they have been found to have close relations to representation stability of (co)homological groups of many algebraic or topological structures; see for instance \cite{FI_moduleOverNoetherianRing}. In particular, Sam and Snowden firstly proved that every finitely generated projective $\FI$-module over a field of characteristic 0 is also injective. Gan and Li introduce the coinduction functor for $\FI$-modules in \cite{CoinductionOfFI} and give another proof of this fact. Furthermore, they give a homological explanation for the following crucial result established by Church, Ellenberg and Farb in \cite{FI_modulesAndStability}: an $\FI$-module over a field of characteristic 0 is finitely generated if and only if the corresponding sequence of representations of symmetric groups is representation stable. In \cite{AppOfNakayama}, Gan, Li and Xi introduce the Nakayama functor and prove that the Serre quotient of the category of finitely generated modules of a family of infinite combinatorial categories (including the category $\FI$ and the category $\mathrm{VI}$ of finite dimensional vector spaces over a finite field and linear injections) over a field of characteristic 0 by the category of finitely generated torsion modules is equivalent to the category of finite dimensional modules.

The product category of several copies of $\FI$ has also been studied; see for instance \cite{CategoriesOfFI_type, FIm_Module}. It is found that many properties such as the local Noetherianity, representation stability, and polynomial growth of Hilbert functions can be extended from $\FI$ to its product categories. Thus we wonder whether the product category $\FIM$ is also locally self-injective over fields of characteristic 0, and whether the above equivalence of categories still holds. The main goal of this paper is to answer these questions.

\subsection{Notations}

Before describing the main results, let us introduce some notations. Throughout this paper let $m$ be a positive integer, and let $\FI^m$ be the category whose objects are $m$-tuples of finite sets $\mathbf{S} = (S_1, \, \ldots, \, S_m)$, and morphisms from an object $\mathbf{S}$ to another object $\mathbf{S'}$ are maps $\bff = (f_1, \, \ldots, \, f_m)$ such that each $f_i: S_i \to S_i'$ is an injection. When $m = 1$, $\FI^m$ coincides with $\FI$. For brevity, We denote by $\mc$ for $\FI$ and $\mcc$ for $\FI^m$.

Let $k$ be a field of characteristic 0. A \textit{representation} of $\mcc$, or a $\mcc$-\emph{module}, is a covariant functor $V$ from $\mcc$ to $k \Mod$, the category of vector spaces over $k$. The value of a representation $V$ on an object $\arr{S}$ is denoted by $V(\arr{S})$. We denote by $\mcc \Mod$ the category of all representation of $\mcc$. It is well known that $\mcc \Mod$ is an abelian category, and it has enough projectives. In particular, for an object $\arr{S}$ in $\mcc$, the $k$-linearization of the representable functor $\mcc(\arr{S}, -)$ is a projective $\mcc$-module. We denote it by $M(\arr{S})$, and we say that a $\mcc$-module is a \textit{free module} if it is a direct sum of several $M(\arr{S})$ up to isomorphism.

Sometimes it is more convenient to view a representation of $\mcc$ as a module over the non-unital $k$-algebra $k\mcc$ called the \textit{category algebra} (for a definition, see \cite{StandardStratifications}). Thus, terminologies from module theory, such as finitely generated modules and finitely presented modules, can be applied. We denote by $\mcc \module$ the category of all \textit{finitely generated} $\mcc$-modules and $\mcc \fdmod$ the category of all \textit{essentially finitely dimensional} $\mcc$-modules, where a $\mcc$-module $V$ is essentially finite dimensional if its values on all objects are finite dimensional and up to isomorphism there are only finitely many objects $\arr{S}$ of $\mcc$ such that $V(\arr{S}) \neq 0$. \footnote{If we take a skeleton subcategory $\mathcal{D}$ of $\mcc$, then essentially finite dimensional $\mcc$-modules correspond exactly to finite dimensional $\mathcal{D}$-modules.} They are also abelian categories; see \cite{AppOfNakayama, FIm_Module}.

Let $V$ be a $\mcc$-module, and $\arr{S}$ be an object of $\mcc$. An element $v \in V(\arr{S})$ is \textit{torsion} if there exists a morphism $\alpha: \arr{S} \to \arr{T}$ such that $\alpha \cdot v = 0$. Torsion elements in $V$ form a submodule of $V$, denoted by $V_T$. Note that the assignment $V \to V_T$ is a left exact functor from $\mcc \Mod$ to itself, and the category $\mcc \Mod^{\mathrm{tor}}$ of torsion modules is also abelian.

\subsection{Main results and strategy}

The first main result of this paper establishes the local self-injectivity of $\mcc$ over a field of characteristic 0. That is,

\begin{theorem}
Let $V_i$, $1 \leqslant i \leqslant m$, be finitely generated injective $\mc$-modules over $k$. Then the external tensor product $V_1 \boxtimes V_2 \boxtimes \ldots \boxtimes V_m$ is a finitely generated injective $\mcc$-modules. In particular, every finitely generated projective $\mcc$-modules is also injective.
\end{theorem}

Finitely generated injective $\mc$-modules have been classified in \cite{GL_equivariantModules}. Explicitly, if $V$ is a finitely generated injective $\mc$-module, then $V$ is a direct sum of an essentially finite dimensional injective $\mc$-module and a finitely generated projective $\mc$-module. It is easy to show the conclusion of the above theorem for the case that each $V_i$ is essentially finite dimensional. In the case that a certain $V_i$ is projective, we use the coinduction functor for $\mc$-modules introduced in \cite{CoinductionOfFI} and show that it commutes with external tensor products, and hence deduce the desired result.

We also slightly modify the technique described in \cite{AppOfNakayama} and prove the following result.

\begin{theorem}
One has the following equivalence
\[
\mcc \module / \mcc \module^{\mathrm{tor}} \xrightarrow{\sim} \mcc \fdmod
\]
which is induced by the Nakayama functor.
\end{theorem}

\begin{remark} \normalfont
The Serre quotient category $\mcc \module/ \mcc \module^{\mathrm{tor}}$ is equivalent to the category of ``finitely generated" sheaves of $k$-modules over the opposite category of $\mcc$ equipped with the atomic topology, and is equivalent to the category of ``finitely generated" discrete representations of the topological group $G^m$, where $G = \mathrm{Aut}(\mathbb{N})$ is the group of bijections between $\mathbb{N}$, equipped with the topology inherited from the canonical product topology. For details, see \cite{SheavesOverCat}.
\end{remark}

\section{Preliminaries}

In this section we describe some preliminary results which will be used later.

\subsection{General facts}

Since most results in this subsection are either well known or have been established in literature, we omit proofs.

\begin{lemma}[{\cite[Theorem 1.1]{FIm_Module}}] \label{lemma_locally_noetherian_of_FIm}
The category $\mcc$ is locally Noetherian over a commutative Noetherian coefficient ring; that is, every submodule of a finitely generated $\mcc$-module is finitely generated.
\end{lemma}

The following lemma is standard (see \cite[page 39]{weibelHomologicalAlgebra}).

\begin{lemma}[Baer’s criterion] \label{lemma_Baer_criterion}
Let $V$ be a $\mcc$-module. Suppose that for all $\arr{S} \in \obj(\mcc)$ and for all $\mcc$-submodules $U$ of $\cce{\arr{S}}$ , every homomorphism $U \to V$ can be extended to a homomorphism $\cce{\arr{S}} \to V$. Then $V$ is injective in $\mcc \Mod$.
\end{lemma}

As an immediate corollary, we have:

\begin{corollary} \label{coro_injModmod}
Let $V$ be a finitely generated $\mcc$-module. Then $V$ is injective in $\mcc \Mod$ if and only if it is injective in $\mcc \module$.
\end{corollary}

\begin{proof}
This follows immediately from Lemmas \ref{lemma_locally_noetherian_of_FIm} and \ref{lemma_Baer_criterion}.
\end{proof}

The following useful lemma is a well known result from homological algebra.

\begin{lemma}[{\cite[Proposition 2.3.10]{weibelHomologicalAlgebra}}] \label{lemma_right_adjoint_preserves_injectives}
Let $\mathcal{A}$ and $\mathcal{B}$ be two additive categories. If an additive functor $R: \mathcal{B} \to \mathcal{A}$ is right adjoint to an exact functor $L: \mathcal{A} \to \mathcal{B}$ and $I$ is an injective object in $\mathcal{B}$, then $R(I)$ is an injective object in $\mathcal{A}$.
\end{lemma}

The following lemma classifies all finitely generated projective $\mcc$-modules.

\begin{lemma} \label{lemma_f.g._projective_module_is_summand_of_free_module}
Any finitely generated projective $\mcc$-module is a direct summand of some free $\mcc$-module.
\end{lemma}

\begin{proof}
Let $P$ be a finitely generated projective $\mcc$-module and $n$ be a positive integer. Suppose that $v_i \in P(\arr{S}_i)$, $i=1,\ldots,n$, form a set of generators of $P$ where $\arr{S}_i \in \obj(\mcc)$. Then there is an obvious surjection
\[
\bigoplus_{i=1}^n M(\arr{S}_i) \to P.
\]
Since $P$ is projective, the above surjection splits. Thus $P$ is a direct summand of $\bigoplus_{i=1}^n M(\arr{S}_i)$ which is a free $\mcc$-module.
\end{proof}

Now we introduce the \textit{external tensor product}, constructing a $\mcc$-module from a family of $\mc$-modules. Given $\mc$-modules $V_1,\cdots,V_m$, we define their external tensor product to be the $\mcc$-module
\[
V_1\boxtimes \cdots \boxtimes V_m: \mcc \to \kMod
\]
such that
\[
(V_1\boxtimes \cdots \boxtimes V_m)(S_1,\cdots,S_m) \ce V_1(S_1) \otimes_k \cdots \otimes_k V_m(S_m)
\]
and
\[
(V_1\boxtimes \cdots \boxtimes V_m)(\alpha_1,\cdots,\alpha_m) \ce V_1(\alpha_1) \otimes_k \cdots \otimes_k V_m(\alpha_m)
\]
where $S_i \in \obj(\mc)$ and $\alpha_i \in \Mor{\mc}$ for $i = 1, \cdots, m$.

\begin{lemma} \label{lemma_external_tensor_preserves_direct_sum}
Let $V_i$, $1 \leqslant i \leqslant m$, and $U$ be $\mc$-modules. Then
\[
V_1 \boxtimes \ldots \boxtimes V_{i-1} \boxtimes (V_i \oplus U) \boxtimes V_{i+1} \boxtimes \ldots \boxtimes V_n \cong (V_1 \boxtimes \ldots \boxtimes V_m) \oplus (V_1 \boxtimes \ldots \boxtimes V_{i-1} \boxtimes U \boxtimes V_{i+1} \boxtimes \ldots \boxtimes V_m)
\]
as $\mcc$-modules.
\end{lemma}

\begin{proof}
This follows immediately from the fact that tensor product commutes with direct sums.
\end{proof}

\subsection{Restriction to finite subcategories}

In this subsection we consider the truncations of $\mcc$-modules. Note that there is a preorder on the class of objects of $\mcc$: for any two objects $\arr{S}$ and $\arr{T}$, we define $\arr{S} \leq \arr{T}$ if and only if $\mcc(\arr{S}, \arr{T}) \neq \emptyset$, or equivalently, for $1 \leqslant i \leqslant m$, one has $|S_i| \leqslant |T_i|$. Fixing an object $\arr{t} \in \obj(\mcc)$, we define $\mcct$ to be the full subcategory of $\mcc$ with objects $\arr{S}$ such that $\arr{S} \leq \arr{t}$.

\begin{definition}
Denote by $\jmath : \mcct \to \mcc$ the inclusion functor. We define the \textit{pullback functor}
\[
\jmath^* : \mcc \Mod \to \mcc_{\leq \arr{t}} \Mod, \quad V \mapsto V \circ \jmath
\]
and the \textit{pushforward functor}
\[
\jmath_* : \mcc_{\leq \arr{t}} \Mod \to \mcc \Mod
\]
such that for $W \in \mcct \Mod$ and $\arr{S} \in \obj(\mcc)$, one has
\[
\jmath_*(W)(\arr{S}) =
\begin{cases}
0, & \arr{S} \nleq \arr{t} \\
W(\arr{S}), & \arr{S} \leq \arr{t}
\end{cases}.
\]
We call $\jmath^*(V)$ the \textit{truncated module} of $V$ with respect to the full subcategory $\mcc_{\leq \arr{t}}$.
\end{definition}

One checks that $\jmath_*$ is the right adjoint functor of $\jmath^*$, and both functors are exact. Moreover, we have:

\begin{lemma} \label{lemma_extOfJmath}
For any $V \in \mcc \Mod$ and $W \in \mcc_{\leq \arr{t}} \Mod$, one has
\[
\Ext^i_{\mcct}(\jmath^*(V), W) \cong \Ext^i_{\mcc}(V,\jmath_*(W))
\]
for all $i \geq 1$.
\end{lemma}

\begin{proof}
Observe that $\jmath^*$ is an exact functor and preserves projective modules. Consequently, $\jmath^*$ preserves projective resolutions of $V$, and the desired result follows from the Eckmann-Shapiro lemma; see \cite[Corollary 2.8.4]{RepAndCoho}.
\end{proof}

We say that a $\mcc$-module $V$ is \textit{bounded} if there exists some object $\arr{t} \in \obj(\mcc)$, called an \textit{upbound} of $V$, such that $V(\arr{S}) = 0$ for all objects $\arr{S} \nleq \arr{t}$. Clearly, $V$ is bounded if and only if there exists an object $\arr{t}$ such that $\jmath_*\jmath^*(V) = V$. Moreover, $V$ is essentially finite dimensional if and only if $V$ is bounded and $V(\arr{S})$ is finite dimensional for every object $\arr{S}$.

\begin{lemma} \label{lemma_bounded_injective_FIm-module_restricts_to_injective_module}
Let $E$ be a bounded $\mcc$-module with an upbound $\arr{t}$. Then $\jmath^*(E)$ is an injective $\mcct$-module if and only if $E$ is an injective $\mcc$-module.
\end{lemma}

\begin{proof}
Let $V$ be an arbitrary $\mcct$-module. By Lemma \ref{lemma_extOfJmath}, we have
\[
\Ext_{\mcct}^1 (V,\jmath^*(E)) \cong \Ext_{\mcct}^1 (\jmath^*\jmath_*(V),\jmath^*(E)) \cong \Ext_{\mcc}^1 (\jmath_*(V),\jmath_*\jmath^*(E)) \cong \Ext_{\mcc}^1 (\jmath_*(V),E),
\]
where the third isomorphism holds because $E \cong \jmath_*\jmath^*(E)$ since it is bounded. From these isomorphisms one deduces the if direction.

Conversely, suppose that $\jmath^*(E)$ is an injective $\mcct$-module, and let $V$ be a $\mcc$-module. We have a natural short exact sequence
\[
0 \to V_{\nleq \arr{t}} \to V \to V_{\leq \arr{t}} \to 0
\]
where $V_{\nleq \arr{t}}$ is the submodule of $V$ such that $V_{\nleq \arr{t}} (\arr{S}) = V(\arr{S})$ whenever $\arr{S} \nleq \arr{t}$ and $V_{\nleq \arr{t}} (\arr{S}) = 0$ otherwise. It induces a long exact sequence
\[
\cdots \to \Hom_{\mcc}(V_{\nleq \arr{t}},E) \to \Ext^1_{\mcc}(V_{\leq \arr{t}}, E) \to \Ext^1_{\mcc}(V,E) \to \Ext^1_{\mcc}(V_{\nleq \arr{t}},E) \to \cdots.
\]
Since $E$ is bounded by $\arr{t}$, the first and forth terms are 0, and $\jmath_* \jmath^* (E) \cong E$. Thus
\[
\Ext^1_{\mcc}(V,E) \cong \Ext^1_{\mcc}(V_{\leq \arr{t}}, E) \cong \Ext^1_{\mcc}(V_{\leq \arr{t}}, \jmath_* \jmath^* (E)) \cong \Ext^1_{\mcc}(\jmath^*(V_{\leq \arr{t}}), \jmath^*(E)) = 0.
\]
Therefore, $E$ is injective.
\end{proof}

The following lemma tells us that the extension group of two finitely generated $\mcc$-modules can be computed in the category of truncated modules.

\begin{lemma} \label{lemma_extReducedToCt}
Let $V$ and $W$ be finitely generated $\mcc$-modules. Then there exists $\arr{N} \in \obj(\mcc)$, depending on $V$ and $W$, such that for all object $\arr{t} \geq \arr{N}$, one has
\[
\Ext^1_{\mcc}(V, W) \cong \Ext^1_{\mcct}(\jmath^*(V),\jmath^*(W)).
\]
\end{lemma}

\begin{proof}
Since $\mcc$ is locally Noetherian, there exists a projective resolution
\[
\cdots \to P^{-2} \to P^{-1} \to P^{0} \to V \to 0
\]
such that each $P^{-i}$ is finitely generated. In particular, there exists $\arr{N} \in \obj(\mcc)$ such that $P^{-2}$ and $P^{-1}$ are both generated by their values on objects $\arr{S}$ with $\arr{S} < \arr{N}$, which means that $\arr{S} \leq \arr{N}$ and $\arr{S}$ is not isomorphic to $\arr{N}$. Suppose that $\arr{t}$ is an object of $\mcc$ with $\arr{t} \geq \arr{N}$, and let $U$ be the submodule $W_{\nleq \arr{t}}$ of $W$. Denote by $\jmath: \mcct \Mod \to \mcc \Mod$ the inclusion functor. We have a short exact sequence
\[
0 \to U \to W \to \jmath_*\jmath^*(W) \to 0
\]
which induces a long exact sequence
\[
\cdots \to \Ext^1_{\mcc}(V,U) \to \Ext^1_{\mcc}(V,W) \to \Ext^1_{\mcc}(V,\jmath_*\jmath^*(W)) \to \Ext^2_{\mcc}(V,U) \to \cdots
\]
But, for $i = 1$ or 2 one has $\homo{\mcc}{P^{-i}}{U} = 0$ and so $\Ext^i_{\mcc}(V,U) = 0$.
It follows that
\[
\Ext^1_{\mcc}(V,W) \cong \Ext^1_{\mcc}(V,\jmath_*\jmath^*(W)) \cong \Ext^1_{\mcct}(\jmath^*(V),\jmath^*(W))
\]
by Lemma \ref{lemma_extOfJmath}.
\end{proof}

From this lemma we deduce a criterion for the injectivity of $\mcc$-modules.

\begin{corollary} \label{corollary_a_sufficient_condition_for_injective_FIm_module}
Let $E \in \mcc \module$ and suppose that there exists an object $\arr{L}$ such that for all $\arr{t}$ satisfying $\arr{t} \geq \arr{L}$, the truncated module $\jmath^*(E)$ is injective. Then $E$ is injective in $\mcc \Mod$.
\end{corollary}

\begin{proof}
Let $V$ be an arbitrary finitely generated $\mcc$-module. Then by Lemma \ref{lemma_extReducedToCt}, one has
\[
\Ext^1_{\mcc}(V, E) = \Ext^1_{\mcct}(\jmath^*(V),\jmath^*(E))
\]
for $\arr{t} \geq \arr{N}$, where $\arr{N}$ depends on $V$ and $E$. In particular, we can choose a suitable $\arr{t}$ such that $\arr{t} \geq \arr{L}$ and $\arr{t} \geq \arr{N}$. In this case, $\jmath^*(E)$ is injective, so the extension group vanishes. Since $V$ is arbitrary, we deduce that $E$ is injective in $\mcc \module$. Now applying Corollary \ref{coro_injModmod} we complete the proof.
\end{proof}

\begin{remark} \normalfont
The converse statement of this corollary, however, is not valid. That is, if $E$ is an injective $\mcc$-module, there might not exist an object $\arr{L}$ such that $\jmath^*(E)$ is injective in $\mcc_{\leq \arr{t}} \module$ for $\arr{t} \geq \arr{L}$. For example, let $m = 1$. Then the $\mc$-module $E = M([1])$ is injective. However, for any $t > 1$, the truncated module $\jmath^* (E)$ is not an injective $\mc_{\leq t}$-module. To see this, we consider the short exact sequence $0 \to E_{>t} \to E \to \jmath^*(E) \to 0$. Let $W$ be the $\mc$-module such that $W(T) = 0$ if $|T| \neq t-1$, and $W(T) = k \mathrm{Aut}(T)$ if $|T| = t-1$. By a direct computation we obtain that $\Ext^1_{\mc_{\leq t}}(\jmath^*W, \jmath^*(E)) \cong \Ext^1_{\mc}(W, \jmath_* \jmath^*(E)) \neq 0$.
\end{remark}

\section{The coinduction functor on $\mcc \Mod$}

In this section we define the coinduction functor for $\mcc$-modules, extending some results described in \cite{CoinductionOfFI}, where the coinduction functor for $\mc$-modules is introduced.

\subsection{Coinduction functor on $\mc \Mod$} \label{section_restriction_and_coinduction_functor}

It is well known that any ring homomorphism $A \to B$ induces a triple of adjoint functors between $A \Mod$ and $B \Mod$, called \textit{induction}, \textit{restriction} and \textit{coinduction} functors respectively (see \cite[Section 2.8]{RepAndCoho}). In our special setting, the same adjoint triple can also be established even though the category algebras we consider may not be unital. In \cite{FI_moduleOverNoetherianRing}, the authors introduced the self embedding functor on category $\mc$, denoted by $\iota$, which induces a shift functor $\Sigma$ on $\mc \Mod$, a special form of the restriction functor. In \cite{CoinductionOfFI}, the right adjoint of $\Sigma$ (called coinduction functor) was introduced. For the convenience of the reader, we recall some constructions.

Firstly we define the self-embedding functor $\iota$ on $\mc$. For a finite set $S$, we denote $\widehat{S} := S \sqcup \set{*}$. For an element $x \in \widehat{S}$, we denote $S^x := S \setminus \{x\}$. Note that if $x = *$, then $(\widehat{S})^* = S$. For any morphism $\alpha \in \homo{\mc}{S}{T}$, we define $\iota(\alpha)$ to be the morphism in $\homo{\mc}{\widehat{S}}{\widehat{T}}$ acting the same as $\alpha$ on $S \subset \widehat{S}$ and fixing the element $*$.

\begin{definition}[restriction functor on $\mc \Mod$ \cite{FI_moduleOverNoetherianRing}]
We define the \textit{shift functor} on $\mc \Mod$ to be
\[
\Sigma : \mc \Mod \to \mc \Mod
\]
that sends every $\mc$-module $V$ to the $\mc$-module $V \circ \iota$. For a natural transformation $\pi : V \to W \in \Mor{\mc \Mod}$, we define $\Sigma\pi$ as
\[
(\Sigma\pi)_S\ce \pi_{\iota(S)} : (V \circ \iota) (S) \to (W \circ \iota) (S)
\]
where $S \in \obj(\mc)$.
\end{definition}

Let $S$ be a nonempty finite set. For an element $x \in \widehat{S}$, we denote by $\varepsilon_x$ the unique morphism $S \to \widehat{S^x}$ sending $x$, if $x \in S$, to $*$ and fixing all other elements in $S$. Note that if $x = *$, then $\varepsilon_*: S \to \widehat{S^x} = \widehat{S}$ is the canonical inclusion, and otherwise, it is an isomorphism.

The following proposition is a reformulation of \cite[Proposition 2.12]{FI_moduleOverNoetherianRing}.

\begin{proposition}\label{prop_shift_functor_on_free_module}
Let $S$ be a finite set and $x$ an element in $\widehat{S}$.
Then $\Sigma M(S) \cong \bigoplus_{x \in \widehat{S}}\left< \varepsilon_x \right>$ as $\mc$-modules, where $\left<\varepsilon_x\right>$ is the cyclic submodule of $\Sigma M(S)$ generated by $\varepsilon_x$. Moreover, $\left< \varepsilon_x \right>$ is isomorphic to the free $\mc$-module $\cce{S^x}$.
\end{proposition}

The right adjoint of shift functor is defined as follows:

\begin{definition}[Coinduction on $\mc \Mod$]
We define the \textit{coinduction functor} to be the functor
\[
\coind: \mc \Mod \to \mc \Mod
\]
such that for $V \in \mc \Mod$ and $S \in \obj(\mc)$,
\[
\coind(V)(S) \ce \homo{\mc}{\Sigma \cce{S}}{V}
\]
and for a morphism $r : S \to T$, we define
\[
\coind(V)(r): \coind(V)(S) \to \coind(V)(T), \quad \alpha \mapsto (g \mapsto \alpha(gr))
\]
We denote $\coind(V)(r)(\alpha)$ as $r \cdot \alpha$ for abbreviation.
\end{definition}

\begin{proposition}[{\cite[Theorem 1.3]{CoinductionOfFI}}] \label{prop_coinduction_on_free_FI_module}
Let $S$ be a finite set. Then
\[
\coind(\cce{S}) \cong \cce{S} \oplus \cce{\widehat{S}}.
\]
\end{proposition}

Using the coinduction functor, Gan and Li give a new proof for the following classification of finitely generated injective $\mc$-modules, firstly obtained by Sam and Snowden in \cite{GL_equivariantModules}.

\begin{proposition}[{\cite[Theorem 1.7(i)]{CoinductionOfFI}}] \label{prop_classfication_of_f.g._injective_FI_module}
Any finitely generated injective $\mc$-module is a direct sum of an essentially finite dimensional injective $\mc$-module and a finitely generated projective $\mc$-module.
\end{proposition}

\subsection{Coinduction functors on $\mcc \Mod$}

In this subsection we introduce shift functors and coinduction functors on $\mcc \Mod$. Note that there are $m$ distinct self-embedding functors on $\mcc$: for $1 \leq i \leq m$, we define $\iota_i$ to be the endofunctor on $\mcc$ such that for any morphism $f=(f_1,\cdots,f_m)$ in the category $\mcc$, we have $\iota_i(f) = (f_1,\cdots,\iota(f_i),\cdots,f_m)$. Similar to results in Subsection \ref{section_restriction_and_coinduction_functor}, the self-embedding functor $\iota_i$ also induces a corresponding shift functor on $\mcc \Mod$. These functors are firstly introduced in \cite{FIm_Module}. We refer the readers to that article for more details.

\begin{definition}[$i$-th shift functor on $\mcc \Mod$]
We define the $i$-th \textit{shift functor} to be the functor
\[
\Sigma_i: \mcc \Mod \to \mcc \Mod, \quad V \mapsto V\circ\iota_i
\]
where $V \in \mcc \Mod$. For $W \in \mcc \Mod$ and any natural transformation $\pi : V \to W$, we define $\Sigma_i\pi$ as
\[
(\Sigma_i\pi)_\arr{S} \ce \pi_{\iota_i(\arr{S})} : (V \circ \iota_i) (\arr{S}) \to (W \circ \iota_i) (\arr{S})
\]
where $\arr{S} \in \obj(\mcc)$.
\end{definition}

Note that each functor $\Sigma_i$ is an exact functor; see \cite{FIm_Module}. Now we construct the following functors which are natural extensions of the coinduction functor on $\mc \Mod$.

\begin{definition}[$i$-th coinduction functor on $\mcc \Mod$]
We define the $i$-th \textit{coinduction functor} to be the functor
\[
\coind_i: \mcc \Mod \to \mcc \Mod
\]
such that for $V \in \mcc \Mod$ and $\arr{S} \in \obj(\mcc)$,
\[
\coind_i(V)(\arr{S}) \ce \homo{\kmcc}{\Sigma_i \cce{\arr{S}}}{V}
\]
and for morphism $\arr{r} : \arr{S} \to \arr{T}$, we define
\[
 \coind_i(V) (\arr{r}): \coind_i(V)(\arr{S}) \to \coind_i(V)(\arr{T}), \quad \alpha \mapsto (\arr{g} \mapsto \alpha( \arr{gr}))
\]
We denote $\coind_i(V)(\arr{r})(\alpha)$ as $\arr{r} \cdot \alpha$ for abbreviation. Note that $\coind$ coincides with $\coind_i$ when $i=1$ and $m=1$.
\end{definition}

The following Lemma is an extension of \cite[Lemma 4.2]{CoinductionOfFI} and can be proved in a similar way.

\begin{lemma}[adjoint relation of $\Sigma_i$ and $\coind_i$] \label{lemma_adjoint}
The $i$-th restriction functor $\Sigma_i$ is left adjoint to the $i$-th coinduction functor $\coind_i$.
\end{lemma}

\begin{proof}
For an object $\arr{T} = (T_1, \ldots, T_m) \in \obj(\mcc)$, we denote the object $\widehat{\arr{T}} = (T_1, \ldots, T_{i-1}, \widehat{T_i}, T_{i+1}, \ldots, T_m)$. Let
\[
U \ce \bigoplus_{\arr{S},\arr{T} \in \obj(\mcc)} k\mcc(\arr{S},\widehat{\arr{T}})
\]
be a $k$-vector space. It has $\kmcc$-bimodule structure via
\[
\arr{r} \cdot u \ce \iota_i(\arr{r})u, \qquad u \cdot \arr{r} \ce u\arr{r},
\]
for $u \in U$ and $\arr{r} \in \Mor{\mcc}$. There is a $\mcc$-module isomorphism
\[
U \otimes_{\kmcc} V \to \Sigma_i(V), \quad u \otimes v \mapsto V(u)(v)
\]
with its inverse map on $\Sigma_i(V)(\arr{N}) = V(\widehat{\arr{N}})$ for the object $\arr{N} \in \obj(\mcc)$ defined by
\[
\Sigma_i(V)(\arr{N}) \to \arr{e}_{\arr{N}} \cdot U \otimes_{\kmcc} V, \quad v \mapsto \arr{e}_{\widehat{\arr{N}}} \otimes v.
\]

On the other hand, it is easy to see that
\[
U = \bigoplus_{\arr{S} \in \obj(\mcc)} \bigoplus_{\arr{T} \in \obj(\mcc)} \kmcc(\arr{S},\widehat{\arr{T}}) = \bigoplus_{\arr{S} \in \obj(\mcc)} \Sigma_i \cce{\arr{S}}
\]
as $\mcc$-modules. Therefore, one has
\[
\homo{\mcc}{U}{W} \cong \prod_{\arr{S} \in \obj(\mcc)} \homo{\mcc}{\Sigma_i \cce{\arr{S}}}{W}.
\]
as $k$-vector space. It has left $\kmcc$-module structure via the right $\kmcc$-module structure of $U$. But since $V$ is a direct sum of $V(\arr{N})$ for $\arr{N} \in \obj(\mcc)$, the image of any $\mcc$-module homomorphism from $V$ to $\homo{\mcc}{U}{W}$ lies in
\[
\coind_i(W) = \bigoplus_{\arr{S} \in \obj(\mcc)}\homo{\mcc}{\Sigma_i \cce{\arr{S}}}{W}.
\]
It follows that
\[
\homo{\mcc}{V}{\homo{\mcc}{U}{W}} \cong \homo{\mcc}{V}{\coind_i(W)}.
\]
Then by the tensor-hom adjuncation, we have
\[
\homo{\mcc}{\Sigma_i(V)}{W} \cong \homo{\mcc}{U \otimes V}{W} \cong \homo{\mcc}{V}{\homo{\mcc}{U}{W}} \cong \homo{\mcc}{V}{\coind_i(W)}.
\]
\end{proof}

\begin{lemma} \label{lemma_extOfSandQ}
$\Ext^j_{\mcc}(V, \coind_i(W)) \cong \Ext^j_{\mcc}(\Sigma_i(V),W)$ for $V, W \in \mcc \Mod$ and $i, j \geq 1$.
\end{lemma}

\begin{proof}
Since $\Sigma_i$ is exact and preserves projective $\mcc$-modules, it also preserves a projective resolution of $V$. By Eckmann-Shapiro's Lemma and Lemma \ref{lemma_adjoint}, we have the desired result.
\end{proof}

Now we start to show that the coinduction functors commute with external tensor products.

\begin{notation} \label{notation_alphaY}
Let $\alpha: S \to T$ be a morphism in $\mc$. For convenience of notation, we denote $\alpha^{-1}: \widehat{T} \to \widehat{S}$ the map such that for $z \in \widehat{T}$, we have
\[
\alpha^{-1}(z) =
\begin{cases}
\text{the preimage of }z \text{ under } \alpha, &z \in \im \alpha \\
*, &z \in \widehat{T} \setminus \im \alpha
\end{cases}.
\]
Note that $\alpha^{-1}$ is not even a morphism in $\mc$. For an element $y \in \widehat{T}$, we set $x = \alpha^{-1}(y)$. We define $\alpha^y: S^x \to T^y$ to be the morphism in $\mc$ such that  $\alpha^y(a) = \alpha(a)$ for all $a \in S^x$. One checks that $\alpha^y \in \homo{\mc}{S^x}{T^y}$ is the unique morphism with $\varepsilon_y\alpha = \iota(\alpha^y) \varepsilon_x$.
\end{notation}

Recall that Proposition \ref{prop_shift_functor_on_free_module} tells us that the $\mc$-module $\Sigma M(S)$ can be decomposed as the direct sum of cyclic submodules $\langle \varepsilon_x \rangle$. Therefore, analysis of any $\mc$-module homomorphism $\Sigma M(S) \to V$ boils down to analysis of homomorphism $\langle \varepsilon_x \rangle \to V$ which is totally determined by its evaluation on $\varepsilon_x$. This leads to the following notation.

\begin{notation} \label{notation_tilde_vx}
Let $S$ be a finite set, $x$ an element in $\widehat{S}$, $V$ a $\mc$-module, and $v^x$ an element in $V(S^x)$. We define $\overline{v^x}: \Sigma M(S) \to V$ to be the unique $\mc$-module homomorphism in $\coind(V)(S)$ such that
\[
\overline{v^x}(\varepsilon_y) = \delta_{xy}v^x
\]
for all $y \in \widehat{S}$ where $\delta$ is the Kronecker delta.
\end{notation}

\begin{lemma} \label{lemma_moduleStructureOfQV}
Notation as before. Let $V$ be a $\mc$-module, $S$ and $T$ objects in $\mc$, $\alpha: S \to T$ a morphism in $\mc$, $x$ an element in $\widehat{S}$, $y$ an element in $\widehat{T}$, and $v^x$ an element in $V(S^x)$. Then
\[
\overline{v^x} \in \coind(V)(S) = \homo{\mc}{\Sigma M(S)}{V}.
\]
Furthermore, for $x \in \widehat{S}$,
\[
\alpha \cdot \overline{v^x}=
\begin{cases}
\overline{\alpha^{\alpha(x)} \cdot v^x}, &x \in S \\
\sum_{y \in \widehat{T} \setminus \mathrm{im} \alpha} \overline{\alpha^y \cdot v^*}, &x = *
\end{cases}.
\]

\begin{proof}
We show the first case. For $x \in S$ and $y \in \widehat{T}$, we have
\[
(\alpha \cdot \overline{v^x})(\varepsilon_y) = \overline{v^x}(\varepsilon_y\alpha) = \overline{v^x}(\alpha^y \cdot \varepsilon_{\alpha^{-1}(y)}) = \alpha^y\cdot(\overline{v^x}(\varepsilon_{\alpha^{-1}(y)})) = \alpha^y \cdot (\delta_{\alpha(x),y} \, v^x) = \delta_{\alpha(x),y} (\alpha^y \cdot v^x),
\]
where the second identity follows from the property of $\alpha^y$ (See Notation \ref{notation_alphaY}), and the forth identity follows from definition of $\overline{v^x}$ (See Notation \ref{notation_tilde_vx}). In conclusion, $\alpha \cdot \overline{v^x} = \overline{\alpha^{\alpha(x)} \cdot v^x}$.

We then show the second case. For $x = *$ and $y \in \im \alpha$, we have $\alpha^{-1}(y) \neq *$, i.e. $\delta_{*,\alpha^{-1}(y)} = 0$. Then
\[
(\alpha \cdot \overline{v^*})(\varepsilon_y) = \overline{v
^*}(\varepsilon_y\alpha) = \overline{v^*}(\alpha^y \cdot \varepsilon_{\alpha^{-1}(y)}) = \alpha^y\cdot(\overline{v^*}(\varepsilon_{\alpha^{-1}(y)})) = 0.
\]
For $x = *$ and $y \in \widehat{T} \setminus \im \alpha$, we have $\alpha^{-1}(y) = *$. By a similar equation, we have $(\alpha \cdot \overline{v^*})(\varepsilon_y) = \alpha^y \cdot v^*$.
In conclusion, $\alpha \cdot \overline{v^*} = \sum_{y \in \widehat{T} \setminus \im \alpha} \overline{\alpha^y \cdot v^*}$.
\end{proof}
\end{lemma}

The main result of this section is:

\begin{theorem} \label{theorem_Coinduction_Preserves_External_Product}
Let $V_j$ be $\mc$-modules for $1 \leq j \leq m$. Then
\[
\coind_i(V_1\boxtimes\cdots\boxtimes V_m) \cong V_1 \boxtimes \cdots \boxtimes \coind(V_{i})\boxtimes \cdots \boxtimes V_m
\]
as $\mcc$-modules.
\end{theorem}

\begin{proof}
We only prove the case $m = 2$ and $i = 1$, since the general case can be proved in a similar way. Let $V$ and $W$ be $\mc$-modules, and $S$ and $T$ be finite sets. Let $\varphi : \Sigma M(S) \boxtimes M(T) \to V \boxtimes W$ be a homomorphism of $\mc$-module. Since $\Sigma M(S)$ is generated by elements in $\set{\varepsilon_x \mid x\in \widehat{S}}$ and $\cce{T}$ is generated by $e_T$, the identity morphism on $T$, we have that the $\mc$-module homomorphism $\varphi$ is determined by $\varphi(\varepsilon_x \otimes e_T)$ for $x \in \widehat{S}$. Furthermore, $\varphi(\varepsilon_x \otimes e_T)$ is an element of the $k$-space $V(S^x) \otimes_k W(T)$ because
\[
\varphi(\varepsilon_x \otimes e_T) = \varphi(e_{S^x} \cdot \varepsilon_x \otimes e_T \cdot e_T) = \varphi((e_{S^x}, e_T)\cdot(\varepsilon_x \otimes e_T)) = (e_{S^x}, e_T)\cdot\varphi(\varepsilon_x \otimes e_T) \in V(S^x) \otimes_k W(T).
\]

Suppose that $\varphi(\varepsilon_x \otimes e_T)$ can be written explicitly as
\[
\varphi(\varepsilon_x \otimes e_T) = \sum_{i=1}^{r_x} v^x_i \otimes w^x_i
\]
where $v^x_i \in V(S^x)$, $w^x_i \in W(T)$ and $r_x$ is a natural number related to $x$. Without loss of generality, we assume that $\{v^x_i\mid i=1,\cdots,r_x\}$ is $k$-linearly independent in $V(S^x)$ for $x \in \widehat{S}$. We define $k$-space homomorphisms
\[
\theta_{ST}: \homo{\mc^2}{\Sigma M(S) \boxtimes M(T)}{V \boxtimes W} \to \homo{\mc}{\Sigma M(S)}{V} \otimes_k W(T)
\]
by letting
\[
\theta_{ST}(\varphi) = \sum_{x \in \widehat{S}} \sum_{i=1}^{r_x} \overline{v^x_i} \otimes w^x_i.
\]
Note that the linear independence of each $\{v^x_i\mid i=1,\cdots,r_x\}$ implies that of $\{\overline{v^x_i}\mid i=1,\cdots,r_x\}$ for a fixed $x$. Therefore, it is easy to check that $\{\overline{v^x_i}\mid x \in \widehat{S},i=1,\cdots,r_x\}$ is also linearly independent.

We are going to show that $\theta_{ST}$ is bijective. We first show that $\theta_{ST}$ is injective. Suppose that
\[
\theta_{ST}(\varphi) = \sum_{x \in \widehat{S}} \sum_{i=1}^{r_x} \overline{v^x_i} \otimes w^x_i = 0.
\]
Then, by \cite[Theorem 14.5]{AdvancedLinearAlgebra}, $w^x_i = 0$ for all $x \in \widehat{S}$ and $i=1,\cdots,r_x$. Therefore, $\varphi(\varepsilon_x \otimes e_T) = 0$ for all $x \in \widehat{S}$ and $\varphi = 0$, and hence $\theta_{ST}$ is injective. The surjectivity of $\theta_{ST}$ follows immediately from the fact that any $\mc$-module homomorphism $\psi \in \homo{\mc}{\Sigma M(S)}{V}$, by Proposition \ref{prop_shift_functor_on_free_module}, is the sum of $\overline{\psi(\varepsilon_x)}$ where $x \in \widehat{S}$.

Finally we show that $\theta_{ST}$ induces a homomorphism of $\mc^2$-modules
\[
\theta: \coind_1(V  \boxtimes W) \to \coind(V) \boxtimes W.
\]
Notation as before. Let $S', T' \in \obj(\mc)$, $\alpha: S \to S'$ and $\beta: T \to T'$, and $y \in \widehat{S'}$. We have
\begin{equation*}
\begin{split}
((\alpha,\beta)\cdot\varphi)(\varepsilon_y \otimes e_{T'}) &= \varphi(\varepsilon_y\alpha \otimes \beta) \\
&= \varphi(\alpha^y\cdot\varepsilon_{\alpha^{-1}(y)}\otimes \beta) \\
&= (\alpha^y,\beta)\cdot\varphi(\varepsilon_{\alpha^{-1}(y)}\otimes e_T)) \\
&= (\alpha^y,\beta)\cdot \sum_{i=1}^{r_{\alpha^{-1}(y)}} v^{\alpha^{-1}(y)}_i \otimes w^{\alpha^{-1}(y)}_i \\
&= \sum_{i=1}^{r_{\alpha^{-1}(y)}} \alpha^y\cdot v^{\alpha^{-1}(y)}_i \otimes \beta\cdot w^{\alpha^{-1}(y)}_i
\end{split}
\end{equation*}
Let $t_{xi} = \alpha^{\alpha(x)} \cdot v^x_i$ for $x \in S$ and $t^y = \alpha^y \cdot v^*$ for $y \in \widehat{S'}\setminus \im \alpha$. By the definition of $\theta_{S'T'}$ together with the above argument, we have
\begin{equation*}
\begin{split}
\theta_{S'T'}[(\alpha,\beta)\cdot\varphi] &= \sum_{y\in \widehat{S'}\setminus \im \alpha}\sum_{i=1}^{r_{*}} \overline{t^y} \otimes\beta\cdot w^*_i + \sum_{x \in S}\sum_{i=1}^{r_x} \overline{t_{xi}} \otimes \beta\cdot w^x_i \\
&= \sum_{i=1}^{r_{*}} \alpha\cdot\overline{v^*}\otimes\beta\cdot w^*_i + \sum_{x \in S}\sum_{i=1}^{r_x} \alpha\cdot\overline{v^x_i} \otimes \beta\cdot w^x_i \\
& = \sum_{x \in \widehat{S}}\sum_{i=1}^{r_x} \alpha\cdot\overline{v^x_i} \otimes \beta\cdot w^x_i \\
&= (\alpha,\beta)\cdot\theta_{ST}(\varphi)
\end{split}
\end{equation*}
The second identity follows from Lemma \ref{lemma_moduleStructureOfQV}. In conclusion, the induced map $\theta$ is a homomorphism of $\mc^2$-modules, hence an isomorphism.
\end{proof}

As an immediate corollary, we get:

\begin{corollary}
Let $\arr{S} = (S_1,\cdots,S_m)$ and $\arr{S'} = (S_1,\cdots,S_{i-1},\widehat{S_i},S_{i+1},\cdots,S_m)$ be objects in $\mcc$. Then
\[
\coind_i(M(\arr{S})) \cong M(\arr{S}) \oplus M(\arr{S'})
\]
\begin{proof}
This follows immediately from Theorem \ref{theorem_Coinduction_Preserves_External_Product}, Proposition \ref{prop_coinduction_on_free_FI_module} and the fact that $M(\arr{S}) \cong M(S_1)\boxtimes\ldots\boxtimes M(S_m)$.
\end{proof}
\end{corollary}

\section{Locally self-injective}
In this section, we show that the external tensor product preserves injectives. As a corollary, every projective $\mcc$-module is injective. It suffices to only consider objects $\numset{n} \ce \set{1,2,\cdots,n}$ of $\mc$ for $n \in \mathbb{N}$. By convention, we set $\numset{0} = \emptyset$.

We first prove a general result.

\begin{lemma} \label{lemma_Tensor_Product_Preserves_Injective_Finitely_Dimensional_Modules}

Let $A = A_1 \otimes_k \cdots \otimes_k A_m$ be the tensor product of finite dimensional $k$-algebras $A_i$, and let $E_i$ be finitely dimensional injective left $A_i$-modules for $i=1,\cdots,m$. Let $E := E_1 \otimes_k \cdots \otimes_k E_m$ be the left $A$-modules with componentwise module action. Then $E$ is an injective left $A$-module.

\begin{proof}
We first show that $DE \cong DE_1 \otimes_k \cdots \otimes_k DE_m$ as right $A$-modules where $D := \homo{k}{-}{k}$ is the standard duality. By \cite[Theorem 14.9]{AdvancedLinearAlgebra}, there is an $k$-space isomorphism \
\[
\tau : DE_1 \otimes_k \cdots \otimes_k DE_m \to DE
\]
with
\[
\tau(f_1\otimes\cdots\otimes f_m)(e_1\otimes\cdots\otimes e_m)=f_1(e_1)\cdots f_m(e_m)
\]
where $f_i \in DE_i$ and $e_i \in E_i$. One checks that the map $\tau$ is actually a right $A$-module isomorphism.

We then show that $DE$ is injective right $A$-module. Since $E_i$ is an injective left $A_i$-module, the property \cite[Theorem I.5.13]{ElementsOfRepTheoryOfAlgebras} of the standard duality ensures that $DE_i$ is a projective right $A_i$-module. Therefore $DE_i$ is a direct summand of some free right $A_i$-module. Since tensor product preserves direct sum, it is easy to see that $DE_1 \otimes_k \cdots \otimes_k DE_m$ is a direct summand of some free right $A$-module, and hence projective. By the isomorphism $\tau$, the right $A$-module $DE$ is projective. By the property \cite[Theorem I.5.13]{ElementsOfRepTheoryOfAlgebras} of standard duality, $D^2E$ is an injective left $A$-module. By the basic property \cite[Definition I.2.9]{ElementsOfRepTheoryOfAlgebras} of standard duality, $D^2 E \cong E$. Therefore $E$ is an injective left $A$-module.
\end{proof}
\end{lemma}

\begin{lemma} \label{lemma_Injectivity_of_Finitely_Dimensional_injective_Modules_Tensor_M0}
Let $V_i$ be a $\mc$-module which is either $M(\numset{0})$ or an essentially finite dimensional injective $\mc$-module for $i = 1,\cdots,m$. Then $V:= V_1\boxtimes\cdots\boxtimes V_m$ is an injective $\mcc$-module.
\end{lemma}

\begin{proof}
If $V_i$ is an essentially finite dimensional injective $\mc$-module, then it is bounded, and we set $n_i$ to be an upbound of $V_i$. If $V_i = M(\numset{0})$, we set $n_i = 0$. Let $\mathbf{N} = (\numset{n_1},\cdots,\numset{n_m})$, which is an object of $\mcc$. Let $\arr{t} = ([t_1],\cdots,[t_n])$ be any object of $\mcc$ satisfying $\arr{t} \geq \mathbf{N}$. It is easy to see that $k\mcc_{\leq \arr{t}}$ is Morita equivalent to $k\mathcal{D}_{\leq t_1} \otimes \ldots \otimes k\mathcal{D}_{\leq t_m}$, where $\mathcal{D}$ is the full subcategory of $\mc$ consisting of objects $[n]$, $n \in \mathbb{N}$. Furthermore, if let $W_i$ be the truncated module of $V_i$ with respect to the subcategory $\mathcal{D}_{\leq t_i}$, then the truncated module of $V$ with respect to the product category $\mathcal{D}_{\leq t_1} \times \ldots \times \mathcal{D}_{\leq t_m}$ is precisely $W_1 \boxtimes W_2 \boxtimes \ldots \boxtimes W_m$. Consequently, by Corollary \ref{corollary_a_sufficient_condition_for_injective_FIm_module}, it suffices to show that $W_1 \boxtimes \ldots \boxtimes W_m$ is an injective $k\mathcal{D}_{\leq t_1} \otimes \ldots \otimes k\mathcal{D}_{\leq t_n}$-module. By Lemma \ref{lemma_Tensor_Product_Preserves_Injective_Finitely_Dimensional_Modules}, we only need to check that each $W_i$ is an injective $k\mathcal{D}_{\leq t_i}$-module.

If $V_i = M(\numset{0})$, then by \cite[Lemma 3.1]{CoinductionOfFI} $W_i$ is injective. If $V_i$ is an essentially finite dimensional injective module, by Lemma \ref{lemma_bounded_injective_FIm-module_restricts_to_injective_module}, $W_i$ is injective as well. In either case we show that $W_i$ is injective, and the conclusion follows.
\end{proof}

Now we are ready to prove the main result of this section.

\begin{theorem} \label{theorem_injectivity_of_external_tensor_of_injective_modules}
Suppose that every $V_i$ is a finitely generated injective $\mc$-module. Then $V= V_1\boxtimes\cdots\boxtimes V_m$ is an injective $\mcc$-module.
\end{theorem}

\begin{proof}
By Proposition \ref{prop_classfication_of_f.g._injective_FI_module}, Lemma \ref{lemma_external_tensor_preserves_direct_sum} and Lemma \ref{lemma_f.g._projective_module_is_summand_of_free_module}, we can assume that $V_i$ is either an essentially finite dimensional injective $\mc$-module or a free $\mc$-module of the form $M(S)$ where $S \in \obj(\mc)$. We prove the conclusion for $m = 2$ since the general case can be established in a similar way. There are several cases:

(1) Both $V_1$ and $V_2$ are essentially finite dimensional injective $\mc$-modules. The conclusion follows from Lemma  \ref{lemma_bounded_injective_FIm-module_restricts_to_injective_module} and Lemma \ref{lemma_Injectivity_of_Finitely_Dimensional_injective_Modules_Tensor_M0}.

(2) $V_1 \cong M(S)$ and $V_2$ is essentially finite dimensional; or dually, $V_1$ is essentially finite dimensional and $V_2 \cong M(S)$.

(3) $V_1 \cong M(S)$ and $V_2 \cong M(T)$.

We prove the second case, and the third one can be shown by a similar way. The proof is based on an induction on the cardinality $|S|$. Without loss of generality we assume that $V_1 \cong M(S)$ and $V_2$ is essentially finite dimensional. If $|S| = 0$, the conclusion follows from Lemma \ref{lemma_Injectivity_of_Finitely_Dimensional_injective_Modules_Tensor_M0}. Now suppose that the conclusion holds for $|S| \leqslant n$ and consider the situation that $|S| = n+1$. By the induction hypothesis, $M(\numset{n}) \boxtimes V_2$ is injective. By Lemma \ref{lemma_right_adjoint_preserves_injectives} the functor $\coind_1$ preserves injective objects, so $\coind_1(M(\numset{n})\boxtimes V_2)$ is injective as well. By Theorem \ref{theorem_Coinduction_Preserves_External_Product} and Proposition \ref{prop_coinduction_on_free_FI_module}, $M(\numset{n+1})\boxtimes V_2$ is a direct summand of $\coind_1(M(\numset{n})\boxtimes V_2)$, and hence is  injective. Clearly, $M(S)\boxtimes V_2 \cong M([n+1]) \boxtimes V_2$ is injective. This finishes the proof.
\end{proof}

As an immediate result, we obtain the following corollary.

\begin{corollary} \label{corollary_finitely_generated_projective_FIm_module_is_injective}
Every finitely generated projective $\mcc$-module is injective.
\end{corollary}

\begin{proof}
This follows from Lemma \ref{lemma_f.g._projective_module_is_summand_of_free_module}, Theorem \ref{theorem_injectivity_of_external_tensor_of_injective_modules} and the fact that $M(\arr{S}) \cong M(S_1)\boxtimes\ldots\boxtimes M(S_m)$ for $\arr{S} = (S_1, \ldots, S_m) \in \obj(\mcc)$.
\end{proof}

\section{Nakayama functor and Serre quotient}

In this section we apply the general results in \cite{AppOfNakayama} to establish an equivalence of categories between the Serre quotient category $\mcc \module /\mcc \module^{\mathrm{tor}}$ and the category $\mcc \fdmod$ of essentially finite dimensional $\mcc$-modules.

Let $V$ be a $\mcc$-module and $\arr{S} \in \obj(\mcc)$. An element $v$ in $V(\arr{S})$ is a \emph{torsion element} if there exists a morphism $\alpha: \arr{S} \to \arr{T}$ such that $V(\alpha)(v) = 0$. We define the \textit{torsion part} $V_T$ of $V$ to be the submodule of $V$ consisting of all torsion elements and the \textit{torsion free part} $V_F$ of $V$ to be the quotient module $V/V_T$. We say that $V$ is a \textit{torsion module} (resp., a \textit{torsion free module}) if and only if its torsion free part (resp., torsion part) is zero. We denote by $\mcc \module^{\mathrm{tor}}$ the category of finitely generated torsion $\mcc$-modules, which is an abelian subcategory of $\mcc \module$.

It is easy to see that $\mcc$ is an EI-category (here EI means that every endomorphism is an isomorphism). It is \textit{essentially inwards finite} (that is, for each $\arr{T} \in \obj(\mcc)$, there are only finitely many $\arr{S} \in \obj(\mcc)$ up to isomorphism such that $\mcc (\arr{S}, \arr{T})$ is nonempty), and \textit{hom-finite} (that is, $\homo{\mcc}{\arr{S}}{\arr{T}}$ is a finite set for all objects $\arr{S}$ and $\arr{T}$). Moreover, the category $\mcc$ is \textit{locally Noetherian} by Lemma \ref{lemma_locally_noetherian_of_FIm} and \textit{locally self-injective} by Corollary \ref{corollary_finitely_generated_projective_FIm_module_is_injective}.

By a general result in \cite[Section 2]{AppOfNakayama}, there is a \textit{standard duality functor} $D \ce \homo{k}{-}{k}$:
\[
\mcc \fdmod \xrightarrow{D} ({\mcc})^{\mathrm{op}} \fdmod.
\]
We also have a pair of contravariant functors:
\[
\homo{\mcc}{-}{\kmcc}: \mcc \module \to ({\mcc})^{\mathrm{op}} \fdmod
\]
and
\[
\homo{({\mcc})^{\mathrm{op}}}{-}{\kmcc}: ({\mcc})^{\mathrm{op}} \fdmod \to \mcc \module.
\]

\begin{definition}
The \textit{Nakayama functor} $\nu$ is defined to be the composition
\[
D \circ \homo{\mcc}{-}{\kmcc}: \mcc \module \to \mcc \fdmod.
\]
The \textit{inverse Nakayama functor} $\nu^{-1}$ is defined to be the composition
\[
\homo{({\mcc})^{\mathrm{op}}}{-}{\kmcc} \circ D : \mcc \fdmod \to \mcc \module.
\]
\end{definition}

For the convenience of the reader, we recall the definition of Serre quotient categories. Let $\mathcal{A}$ be an abelian category. A \textit{Serre subcategory} of $\mathcal{A}$ is a full subcategory $\mathcal{B}$ of $\mathcal{A}$ which is closed under subobjects, quotient objects and extensions. The \textit{Serre quotient} $\mathcal{A}/\mathcal{B}$ is the category whose object class equals that of $\mathcal{A}$ and for any two objects $X$ and $Y$, the hom-set
\[
\homo{\mathcal{A}/\mathcal{B}}{X}{Y} := \underrightarrow{\mathrm{lim}}\homo{\mathcal{A}}{X'}{Y/Y'}
\]
is the colimit over subobjects $X' \subseteq X$ and $Y' \subseteq Y$ such that $X/X' \in \mathcal{B}$ and $Y' \in \mathcal{B}$.

Now we define $\ker(\nu)$ to be the full subcategory of $\mcc \module$ consisting of objects $V$ such that $\nu(V) = 0$. It is easy to check that this is a Serre subcategory of $\mcc \module$. The following proposition states that $\ker(\nu)$ is exactly the category of all finitely generated torsion $\mcc$-modules.

\begin{proposition} \label{proposition_equality_of_Nakayama_kernel_and_torsion_modules}
Notation as above. One has $\ker(\nu) = \mcc \module^{\mathrm{tor}}$.
\end{proposition}

\begin{proof}
If $V$ is a finitely generated torsion $\mcc$-module, then clearly $\homo{\mcc}{V}{M(\arr{S})} = 0$ for all $\arr{S} \in \obj(\mcc)$ because $M(\arr{S})$ is torsion free. Therefore, $V \in \ker(\nu)$, and hence $\ker(\nu) \supseteq \mcc \module^{\mathrm{tor}}$.

On the other hand, if $V$ is not a torsion $\mcc$-module, then there is a short exact sequence in $\mcc \module$
\[
0 \to V_T \to V \to V_F \to 0.
\]
Note that $V_F$ is nonzero since $V$ is not a torsion module. Thus, by \cite[Lemma 2.6(3)]{FIm_Module}, the map $V_F \to \Sigma_iV_F$ is injective for $i = 1, \cdots, m$ and we obtain a nonzero map $V \to \Sigma_iV_F$ with $\Sigma_iV_F$ torsion free by \cite[Lemma 2.6(4)]{FIm_Module}. By the same argument, we will obtain an injection $\Sigma_iV_F \to \Sigma_j\Sigma_iV_F$. Composing with the map $V \to \Sigma_iV_F$, we obtain an injective map $V \to \Sigma_j\Sigma_iV_F$. Recursively, one can get an injective map $V \to \Sigma^n_1\cdots\Sigma^n_mV_F$ for a sufficiently large $n$. By \cite[Proposition 4.10]{FIm_Module}, $\Sigma^n_1\cdots\Sigma^n_mV_F$ is relative projective (See \cite[Definition 4.1]{FIm_Module}) and is projective since $k$ is field of characteristic 0 (for a reason, see \cite[Section 5.1]{FIm_Module}). Therefore, there exists a finitely generated projective $\mcc$-module $P = \Sigma^n_1\cdots\Sigma^n_mV_F$ such that $\homo{\mcc}{V}{P} \neq 0$. Thus, $V \notin \ker(\nu)$.
\end{proof}

Consequently, applying \cite[Theorem 3.6]{AppOfNakayama} we obtain the following result:

\begin{theorem}
The Nakayama functor $\nu$ induces an equivalence of categories
\[
\mcc \module/\mcc \module^{\mathrm{tor}} \xrightarrow{\sim} \mcc \fdmod.
\]
\end{theorem}

\begin{proof}
This follows immediately from Proposition \ref{proposition_equality_of_Nakayama_kernel_and_torsion_modules} and \cite[Theorem 3.6]{AppOfNakayama}.
\end{proof}


\begin{thebibliography}{99}

\bibitem{ElementsOfRepTheoryOfAlgebras}
I. Assem, A. Skowronski, D. Simson,
\textit{Elements of the Representation Theory of Associative Algebras Vol. 1},
Cambridge University Press, Cambridge, 2006. % doi: https://doi.org/10.1017/CBO9780511614309.

\bibitem{RepAndCoho}
D. J. Benson,
\textit{Representations and Cohomology Vol. 1},
Cambridge University Press, Cambridge, 1991. % doi: https://doi.org/10.1017/CBO9780511623615.

\bibitem{FI_moduleOverNoetherianRing}
T. Church, J. Ellenberg, B. Farb, R. Nagpal,
FI-modules over Noetherian rings,
\textit{Geom. Topol.} \textbf{18} (2014), 2951--2984.   arXiv:1210.1854v2.

\bibitem{FI_modulesAndStability}
T. Church, J. Ellenberg, B. Farb,
FI-modules and stability for representations of symmetric groups,
\textit{Duke Math. J.} \textbf{164} (2015) 1833–1910, arXiv :1204.4533.

\bibitem{SheavesOverCat}
Z. Di, L. Li, L. Liang, F. Xu,
Sheaves over categories with atomic topology and discrete representations of topological groups,
Preprint.

\bibitem{CategoriesOfFI_type}
N. Gadish,
Categories of FI type: a unified approach to generalizing representation stability and character polynomials,
\textit{J. Algebra} \textbf{480} (2017) 450–486, arXiv :1608.02664.

\bibitem{CoinductionOfFI}
W. L. Gan, L. Li,
Coinduction functor in representation stability theory,
\textit{J. Lond. Math. Soc} \textbf{92} (2015), 689-711. arXiv:1502.06989.

\bibitem{AppOfNakayama}
W. L. Gan, L. Li, C. Xi,
An application of Nakayama functor in representation stability theory,
\textit{Indiana Univ. Math. J.} \textbf{69} (2020), 2325–2338. arXiv:1710.05493v1.

\bibitem{FIm_Module}
L. Li, N. Yu,
FI$^m$-modules over Noetherian rings,
\textit{J. Pure Appl. Algebra} \textbf{223} (2018), 3436-3460. arXiv:1705.00876v2.

\bibitem{AdvancedLinearAlgebra}
S. Roman,
\textit{Advanced Linear Algebra},
Springer-Verlag New York, 2008. % doi: https://doi.org/10.1007/978-0-387-72831-5.

\bibitem{GL_equivariantModules}
S. Sam, A. Snowden,
GL-equivariant modules over polynomial rings in infinitely many variables,
\textit{Trans. Amer. Math. Soc.} \textbf{368} (2016), 1097–1158.

\bibitem{weibelHomologicalAlgebra}
C. A. Weibel,
\textit{An Introduction to Homological Algebra},
Cambridge University Press, Cambridge, 1994. % doi: https://doi.org/10.1017/CBO9781139644136.

\bibitem{StandardStratifications}
P. Webb,
Standard stratifications of EI categories and Alperin's weight conjecture,
\textit{J. Algebra} \textbf{320} (2008), 4073-4091.

\end{thebibliography}
\end{document}